\documentclass[12pt]{amsart}
\usepackage[a4paper,margin=26mm]{geometry}

\usepackage[T1]{fontenc}
\usepackage[utf8]{inputenc}  % inutile si vous compilez avec LuaLaTeX ou XeLaTeX

\usepackage{lmodern}

\usepackage{amsmath,amssymb}
\usepackage{mathtools}        % charge amsmath en interne, donc redondant mais inoffensif
\usepackage{bm}               % bold math

\usepackage{microtype}        % après les polices, avant hyperref

\usepackage{graphicx}
\usepackage{xcolor}
\usepackage{hyperref}
\definecolor{darkgreen}{rgb}{0,0.4,0}
\definecolor{BrickRed}{rgb}{0.65,0.08,0}
\hypersetup{colorlinks=true} %It is ugly to have the hyperlinks surrounded by a square, this line solves it.
\hypersetup{linkcolor=blue,citecolor=darkgreen,filecolor=BrickRed,urlcolor=darkgreen} 

\usepackage{comment}

\newtheorem{theorem}{Theorem}

\newtheorem{lemma}{Lemma}

\newtheorem{definition}{Definition}

\numberwithin{equation}{section}

\begin{document}
\title{Combinatorics and Asymptotics of\\ Positive Systems of Linear Catalytic Equations}

% 'Each author has his or her own set of coordinates.'
\author{Cyril Banderier}
%\email{} %no email to avoid automatic spams by robots
\urladdr{https://lipn.fr/{\string~}banderier}
\address{CNRS/Universit\'e Sorbonne Paris-Nord, Villetaneuse, France.}
\author{Michael Drmota}
%\email{} %no email to avoid automatic spams by robots
\urladdr{https://www.dmg.tuwien.ac.at/drmota/}
\address{Institut f\"ur Diskrete Mathematik und Geometrie, TU Wien, Wien, Austria.}
%

% 'MSC classification, keywords and grant acknowledgements'
\date{\today} 
\thanks{The first author is supported by the Danube project 
(Scientific \& Technological Cooperation, Austrian/French/Czech OeAD WTZ project, PHC Danube 2025-2026)
and the second author by the Austrian Science Fund (FWF), projects P35016 and 10.55776/F1002}

\begin{abstract}
We provide a complete combinatorial and asymptotic analysis of 
positive linear systems of equations in one catalytic variable that
appear in several combinatorial problems such as in lattice path 
counting or stack-sortable permutation counting.

We show that the corresponding generating functions satisfy a positive polynomial system of equations 
(which is associated to a context-free grammar). 
Furthermore we prove a universal asymptotic behaviour.
\end{abstract}

\maketitle

\begin{center}
  \begin{minipage}[c]{0.8\linewidth}
    % table of contents depth
     \setcounter{tocdepth}{1}
    \tableofcontents  
  \end{minipage}  
\end{center}

\section{Introduction}

Many combinatorial structures resist a direct simple enumeration, 
while, paradoxically, attempting to enumerate a problem that is, a priori,  harder
(a refinement of these structures taking into account one more parameter) 
often helps to get a recursive description which yields a solvable functional equation for the corresponding bivariate generating function.  This additional parameter requires an additional variable (that we will denote by $u$) in the functional equation, 
which is thus called a \textit{catalytic variable}, and the corresponding functional equation is called a \textit{catalytic equation}:
\[
P(x,u,F(x,u),F(x,0)) = 0.
\]
This terminology, natural in chemistry, was popularized in combinatorics in particular by Zeilberger~\cite{Zeilberger2000}.

This is, for example, an approach which was used successfully by Tutte for the enumeration of planar maps~\cite{Tutte1963} and by Knuth for the enumeration of permutations sortable via a stack~\cite{Knuth1}.
Solving such catalytic equations can seem, at first glance, to be quite a challenge; it requires ad-hoc methods 
such as the quadratic method of Tutte and Brown for maps~\cite{BrownTutte},
or the kernel method of Knuth for permutations and lattice paths.
These methods were later generalized beyond equations of degree 2
(see e.g.~\cite{BousquetMelouJehanne2006, BousquetMelouPerkovsek2000, BanderierFlajolet2002, BanderierBousquetMelouDeniseFlajoletGardyGouyouBeauchamps1999}), and to systems \cite{AsinowskiBacherBanderierGittenberger,NotarantonioYurkevich}, 
while also allowing the study of various parameters and structures \cite{AsinowskiBanderier, BanderierGittenberger2006, AguechBanderier}.

The purpose of this contribution is to prove a complete combinatorial and asymptotic analysis of 
positive \emph{linear} systems of equations in \emph{one} catalytic variable. 
We show that the corresponding generating functions satisfy a 
so-called positive polynomial system of equations (which corresponds to a context-free grammar)
and that the coefficients satisfy (under natural conditions) a universal asymptotic behaviour.
This provides a broadly framed response to the asymptotic question raised by Knuth~\cite[Section 2.2.1, Exercise 12]{Knuth1}.

\subsection{Example for a linear catalytic equation}
\label{subsec1.1}

As a first example, let us consider the equation associated to Knuth's seminal problem (we refer to~\cite[Section 2.2.1, Exercises 4 and 11]{Knuth1} for the presentation
of the combinatorial and computer science motivations behind this equation):
\begin{equation}\label{eqfirstexample}
F(x,u) = 1 + x u F(x,u) + x \frac{F(x,u)-F(x,0)}{u},
\end{equation}
where we assume that $F(x,u)$ is a power series -- and it is an easy exercise to show that
there is actually a unique solution with nonnegative coefficients.
If we just multiply this equation by $u$ and set then $u=0$ then we get $0=0$ and, thus, no information.
However, if we regroup~\eqref{eqfirstexample} as
\[
\left( 1 - xu - \frac xu\right) F(x,u) = 1 - x\frac{F(x,0)}{u},
\]
then it is natural to assume that $x$ and $u$ are related by the relation (or kernel equation)
\begin{equation}\label{eqfirsturelation}
 1 - xu - \frac xu = 0 \quad\mbox{or}\quad  u = x(1+u^2).
\end{equation}
Formally we can write $u=u(x)$ as the solution of the equation $u = x(1+u^2)$
(with $u(0) = 0$\footnote{The other solution of the equation (\ref{eqfirsturelation}) does not lead to
a power series solution of (\ref{eqfirstexample}).}) which can be explicitly computed to
$u(x) = ({1-\sqrt{1-4x^2}})/(2x)$.

If we restrict ourselves to $x>0$ und $u> 0$ then $F(x,u) > 0$ and we are left with the relation
\[
1- x\frac{F(x,0)}{u(x)} = 0
\]
and we can compute $F(x,0)$ and $F(x,u)$ as
\[
F(x,0) = \frac{u(x)}{x} = \frac{1-\sqrt{1-4x^2}}{2x^2}
\quad\mbox{and}\quad 
F(x,u) = \frac{1 - x\frac{F(x,0)}{u}}{ 1 - xu - \frac xu}.
\]
Since these functions are power series they constitute the unique solution of (\ref{eqfirstexample}).

Actually, this problem is just a rephrasing of a simple lattice-path counting problem. Let $F_k(x)$ be the generating function of lattice paths (where the variable $x$ marks the total number of steps)
that start at the origin~$(0,0)$ and end at~$(n,k)$, where only up-steps~$(1,1)$ and down-steps~$(1,-1)$ are allowed and all paths stay above the $x$-axis
(these are the ubiquitous \textit{Dyck paths}).
They clearly satisfy the infinite system of equations, 
as seen by splitting any Dyck path before its last step:
\[\begin{cases}
    F_0(x) = 1 + F_1(x) x, \\
    F_1(x) = F_0(x) x + F_2(x) x, \\
    F_2(x) = F_1(x) x + F_3(x) x, \\
%    F_3(x) = F_2(x) x + F_4(x) x, \\
    \phantom{ F_3(x)\,\,\,}\vdots
\end{cases}\]
Clearly, by multiplying the $k$-th equation by $u^k$ and summing it up, we directly obtain~\eqref{eqfirstexample} with $F(x,u) = \sum_{k\ge 0} F_k(x) u^k$.
It is classical that Dyck paths can be directly counted in many ways but note that in addition to the above \emph{infinite} system,
we also have a \emph{finite positive} system (a context-free grammar). Indeed, let $A_0(x)$ denote the
generating function of paths from $(0,0)$ to $(n,0)$ with $n> 0$ that stay strictly 
above the $x$-axis in between. Then we have
\begin{equation}
%\begin{cases}
    A_0(x)= x F_0(x) x,   \qquad 
    F_0(x) = 1 + F_0(x)A_0(x) 
%    \end{cases}
    \label{eqF0}
\end{equation}
and consequently $F_0(x) = (1-\sqrt{1-4x^2})/(2x^2)$ (as before).

We note that $F_0(x)$ and consequently all functions $F_k(x)$ have a 
common dominating square-root singularity at $x_0 = \pm \frac 12$.
Our article shows that these claims hold in fact not just for this simple equation associated to the Dyck case,
but, in full generality, for \emph{systems} of linear catalytic equations of \emph{higher order}.

\subsection{Systems of linear catalytic equations}

In order to define the general problem, we consider the $\ell$-th difference of a function $G(x,u) = \sum_{k\ge 0} G_k(x) u^k$ (with $\ell \ge 1$) as
\[
\Delta^\ell G(x,u) = \frac{G(x,u) - G_0(x) - G_1(x)u - \cdots - G_{\ell-1}(x) u^{\ell-1}}{u^\ell}
= \sum_{k\ge 0} G_{k+\ell}(x) u^k.
\]
Furthermore, we set $\Delta^0 G(x,u) = G(x,u)$.

\begin{definition}
    We say that a set of $d\ge 1$ generating functions $F_1(x,u), \ldots, F_d(x,u)$ 
    satisfies a positive linear system of polynomial catalytic equations of order $L\ge 1$ in one
    catalytic variable $u$ if there are polynomials $P_{i}(x,u)$ and $Q_{i,j,\ell}(x,u)$ 
    $1\le i,j \le d$, $0\le \ell \le L$, with nonnegative coefficients 
    such that
\begin{equation}\label{eqcatalyticsystem}
    F_i(x,u) = P_{i}(x,u) + x \sum_{j=1}^d \sum_{\ell =0}^L Q_{i,j,\ell}(x,u) \Delta^\ell F_j(x,u), \qquad (1\le i \le d).
\end{equation}  
\end{definition}

It is known that such systems of equations can always be solved (see \cite{AsinowskiBacherBanderierGittenberger, NotarantonioYurkevich})
by an extended kernel method (the easiest case being Equation~\eqref{eqfirstexample}). This method shows that all solutions are algebraic functions. 
Thus, it is always possible to get a complete asymptotic analysis for the coefficients $[x^n] F_{i,0}(x)$ by using a proper singularity analysis.
However, if we focus on positive systems (as defined here), it turns out that we
are in a quite special situation. First, the solution function is also the solutions
of a positive polynomial system (Theorem~\ref{Th1}) and, second, there is a universal asymptotic behaviour
in the strongly connected case, where all dominant singularities are of square-root type (Theorem~\ref{Th2}).

This was already proved in the case of one linear equation 
for order $L\le 2$ (see \cite{DrmotaHainzl2023}).
Thus, our results are not only a far-reaching generalization of the results of \cite{DrmotaHainzl2023}
but provide a complete solution for systems of linear catalytic equations.

The main feature of this paper is to present a solution procedure that avoids 
the kernel method but rephrases the problem into a proper weighted lattice-path counting problem. 
In particular, this lattice path problem can then be solved with the help of a strongly connected positive non-linear polynomial system (similarly as for $A_0$ and $F_0$ in the above example). 
Since the solutions of strongly connected positive non-linear polynomial systems have a common dominant square-root singularity we obtain our main result.

If the system is not strongly connected, then the transfer to a positive nonlinear polynomial system still applies, and we also get a characterization of the analytic structure.

We note that the corresponding problem for non-linear systems of catalytic equations is not solved yet. 
Only the case of one equation could be handled (see~\cite{DrmotaNoyYu2022,  DrmotaHainzl2023, SchaefferDuchi2025} for the cases $L=1$ and $L=2$ and~\cite{ContatCurien2025} for all $L\ge 1$).

\section{Results}

We say that a generating function $F(x)$ corresponds to a \textit{finite grammar} 
or, equivalently, satisfies a \textit{positive polynomial system} if there is a polynomial system of equations
$F_i(x) = R_i(x,F_1(x),\ldots, F_d(x))$, $1\le i \le d$,
where $R_i(x,y_1,\ldots,y_d)$ are polynomials with nonnegative coefficients and
where $F(x) = F_1(x)$ is the first component of the (unique) power series solutions.

The first result provides a very strong link to systems of linear catalytic systems.
\begin{theorem}\label{Th1}
    Suppose that $F_i(x,u) = \sum_{k\ge 0} F_{i;k}(x) u^k$ is a solution of a positive linear system of polynomial catalytic equations of the form~\eqref{eqcatalyticsystem}. 
    Then for each $k\ge 0$ the function $F_{i;k}(x)$ corresponds to a finite grammar.
\end{theorem}

It has been shown with different approaches by Popescu (via Artin approximation~\cite{Popescu86}) 
and by Bousquet-M\'elou and Jehanne (via algebraic elimination~\cite{BousquetMelouJehanne2006})
% (see \cite{BousquetMelouJehanne2006, Popescu86} %CastroJimenezPopescuRond2019}) 
that all solutions of (systems of) polynomial catalytic equations are algebraic. Thus, all singularities are algebraic. 
Our article shows that it is possible to get rid of the minus sign involved by design in many of these catalytic equations, and that one then gets a \emph{positive} polynomial system. 
Now, functions solutions of such positive systems have the property that their dominant singularities are quite restricted:
Banderier and Drmota have shown in \cite{BanderierDrmota2015} that the exponents $\alpha$ of the Puiseux expansions at the dominant singularities can only be
of the form $\alpha =  2^{-k}$ for some $k\ge 1$ or of the form $\alpha =  -m 2^{-k}$
for some $m\ge 1$ and $k \ge 0$.
%\[
%\alpha \in \{ 2^{-k} : k \ge 1 \} \cup  \{ -m 2^{-k} :  \}
%\] 
What it is more, many systems have a  strongly connected graph of dependence (in particular, it is the case when the system has a single equation);
it implies a dominant square-root singularity ($\alpha = \frac 12$).
This is what Flajolet and Sedgewick called the Drmota--Lalley--Woods theorem (we refer to the books~\cite{Drmota2009} and~\cite{FlajoletSedgewick2009}
for more on this theorem).

%This singularity occurs for a single (non-linear) equation as well as for so-called %strongly connected systems. 
Combining these results, we obtain a quite universal asymptotic behaviour for positive linear systems of polynomial  \emph{catalytic} equations (under natural combinatorial conditions).
\begin{theorem}\label{Th2}
    Suppose that the functions $F_i(x,u)$, $1\le i\le d$ are the solutions of a positive linear system of polynomial catalytic equations of the form~\eqref{eqcatalyticsystem},
    where the corresponding infinite linear system for the functions $[u^k] F_i(x,u)$ (for $1\le i\le d$ and $k\ge 0$) is strongly connected.
    Then all solution functions $[u^k] F_i(x,u)$ have a common dominating square-root singularity.
\end{theorem}
Furthermore, we also get \textit{universal asymptotic expansions} for the coefficients that are of type 
\[
[x^n] F_{i;k}(x) \sim c_m \rho^{-n} n^{-3/2},
\]
for $n \equiv m \bmod M$ and some constant $c_m$ that depends on $m$ (with $m\in \{0,1, \ldots, M-1\}$),
compare with~\cite{BanderierDrmota2015}.

\subsection*{Plan of the proofs}

The proofs of Theorems~\ref{Th1} and \ref{Th2} follow the main lines that 
are presented in Section~\ref{subsec1.1} (when we dealt with a very simple example). 
However, one has to add some more ideas to make the method work in general:
\begin{enumerate}
\item In a first step the solution functions $F_{i;k}(x) = [u^k] F_i(x,u)$ of
the system (\ref{eqcatalyticsystem}) are interpreted as generating functions
of a weighted coloured lattice path problem with paths that stay above the $x$-axis.
In the case of a system, one has to distinguish between types of paths and 
types (colours) of steps (according to the number of equations).
\item In a second step, one solves the weighted lattice path problem in a direct
way (similarly to the example of Section~\ref{subsec1.1}) by introducing 
so-called \textit{prime walks} (see \cite{BanderierLacknerWallner2020}). 
This leads to an infinite positive non-linear polynomial system of equations (without a catalytic variable).
It is technical, but one can then show that there is a proper finite subsystem from which
all other solution functions can be determined (in the above example, this was the context-free grammar associated to the system~\eqref{eqF0}). 
This proves Theorem~\ref{Th1}.
\item In order to prove Theorem~\ref{Th2} one needs a proper strongly-connected subsystem.
Actually, this is not immediate (the generic case is less obvious than what the system~\eqref{eqF0} may suggest).
The point is that the step sets that can be applied at some level $i$ vary with $i$, however, 
in a monotonely increasing way and they get constant from a certain level on. 
This property is reflected in a monotonicity property of the corresponding generating function.
One can then introduce a collection of new (positive) generating functions as differences of 
these monotone generating functions. And it turns out that, with the help of these
differences, one obtains a positive polynomial subsystem that is strongly connected.
Thus, all appearing generating functions have a dominating square-root singularity 
which proves Theorem~\ref{Th2}.
\end{enumerate}
We start with the analysis of one equation in Section~\ref{sec:oneequation}.
The general case is then discussed in Section~\ref{sec:several}.

Before we start, we want to mention that the enumeration of generalized Dyck paths -- and this is precisely what we will do -- has a quite long history. We just mention the papers \cite{LabelleYeh90,Labelle2,Duchon00}.

\section{One catalytic equation}\label{sec:oneequation}

We first discuss the case of one catalytic equation
\begin{equation}\label{eqoneequation}
    F(x,u) = P(x,u) + x \sum_{\ell = 0}^L Q_{\ell}(x,u) \Delta^\ell F(x,u).
\end{equation}
We first focus on the case $P(x,u) = 1$ and will show later  
how the general case can be reduced to this case (see Section~\ref{sec:reduction}).

\subsection{Weighted lattice paths}

In a first step, we show that the solution of a catalytic equation that is directly related to weighted lattice paths problem. We consider the catalytic equation of the form
\begin{equation}\label{eqoneequationsimple}
    F(x,u) = 1 + x  \sum_{\ell = 0}^L  \sum_{j=0}^{J} Q_{\ell,j}(x) u^j \Delta^\ell F(x,u),
\end{equation}
where the polynomials $Q_{\ell,j}(x)$ have nonnegative coefficients.
If we set
\[
F(x,u) = \sum_{k\ge 0} F_{k,0}(x) u^k
\]
the catalytic equation~\eqref{eqoneequationsimple} transforms 
into the following infinite system:
\begin{align}
    F_{0,0}(x) &= 1 + x \sum_{\ell = 0}^{L} Q_{\ell,j}(x) F_{\ell,0}(x), \nonumber \\
    F_{k,0}(x) &= x \sum_{\ell = 0}^L \sum_{j = 0}^k Q_{\ell,j}(x) F_{\ell+k-j,0}(x), 
    \qquad (0< k < J) \label{eqinfinitesystemsimple}\\
        F_{k,0}(x) &= x \sum_{\ell = }^L \sum_{j = 0}^J Q_{\ell,j}(x) F_{\ell+J-j,0}(x) \qquad (k > J) \nonumber 
\end{align}
We now consider lattice walks that connect the points $(0,k)$ with $(n,0)$ with $k,n\ge 0$
that stay above the $x$-axis. Every path $p$ consists of a sequence of steps
$p = s_1 s_2 \cdots s_M$ (with $M\ge 0$). Formally, the steps are just added up 
\[
(0,k) + s_1 + s_2 + \cdots + s_M = (n,0)
\]
so that the second component of $(0,k) + s_1 + s_2 + \cdots + s_j $ is nonnegative 
for all $0\le j \le M$.

For every level $k$ we define a 
multi-set $\mathcal{S}_k$ of allowed steps that start at this level:
\begin{align}
\mathcal{S}_k &= \bigcup_{\ell = 0}^L \bigcup_{j = 0}^k 
\{ (1+r,\ell-j) : 0\le r \le {\rm deg}Q_{\ell,j}, [x^r] Q_{\ell,j}(x) \ne 0   \} 
\qquad (k < J)  \label{eqSk1}\\
\mathcal{S}_k &= \bigcup_{\ell = 0}^L \bigcup_{j = 0}^J 
\{ (1+r,\ell-j) : 0\le r \le {\rm deg}Q_{\ell,j}, [x^r] Q_{\ell,j}(x) \ne 0   \} 
\qquad (k \ge J).   \label{eqSk2}
\end{align}
Note that in all cases 
\[
1+ r \ge 1 \quad\mbox{and}\quad  -k \le \ell- j \le L.
\]
Thus, every step is going to the right, and if we start at level $k \ge 0$ then
we stay above the $x$-axis if we add a step from $\mathcal{S}_k$.
Further note that the step multi-sets are monotone and stable, that is,
$\mathcal{S}_k \subseteq \mathcal{S}_{k+1}$ for all $k\ge 0$ and 
$\mathcal{S}_k = \mathcal{S}_J$ for all $k\ge J$.

We say that a path $p = s_1s_2\cdots s_M$ is a path from $(k,0)$ to $(0,n)$
is allowed if we have that property that $s_m \in \mathcal{S}_{k_m}$
(for $0\le m\le M$) if  $s_m$ starts at level $k_m$.
Let $\mathcal{P}_{n,k}$ denote the multi-set of allowed paths that connect $(0,0)$ with
$(k,n)$. 

%Without loss of generality, we can also assume that, starting from any level $k$, we can
%pass any given level $\overline k$ (and go back to  $(n,0)$ for some $n\ge 0$) 
%with an allowed path.
%We call such a situation {\it canonical}.

Furthermore, we define weights
\begin{align*}
w(1+r,\ell-j) &= [x^r] Q_{\ell,j}(x), \qquad (0\le \ell \le L,\  0\le j\le J)
\end{align*}
for all possible steps. Note again that the steps are in a multi-set. 
Thus, if a step $(1+r,\ell-j)$ appears several times (according to the non-zero polynomial coefficients
$[x^r] Q_{\ell,j}(x)$ with the same $r$ and the same difference $\ell-j$) then we have to assign the
corresponding weights $[x^r] Q_{\ell,j}(x)$ accordingly.

Finally, if $p = s_1s_2\cdots s_M \in \mathcal{P}_{n,k}$ is an allowed path from $(k,0)$ to $(0,n)$
then the weight of $p$ is defined by 
\[
w(p) = \prod_{m=1}^M w(s_m).
\]
Note again that the same path $p$ might appear several times in $\mathcal{P}_{n,k}$ with
potentially different weights.

The relation between the catalytic equation (\ref{eqoneequationsimple}) and the lattice paths that
have been described here is the following one.

\begin{lemma}\label{Le2.1}
    Suppose that $Q_{\ell,j}(x)$, $0\le \ell \le L$, $0\le j\le J$, are given polynomials
    with nonnegative coefficients. With the help of these polynomials define so-called
    allowed multi-sets of lattice paths $\mathcal{P}_{n,k}$ as described above. 
    Furthermore, set 
    \[
     F_{k,0}(x) = \sum_{n\ge 0} \sum_{p\in \mathcal{P}_{n,k}} w(p)\, x^n
     \]
     and
     \[
     F(x,u) =  \sum_{k\ge 0 } F_{k,0}(x) u^k.
     \]
     Then $F(x,u)$ satisfies then equation (\ref{eqoneequationsimple}) and is also
     the unique power series solution of this equation.
\end{lemma}

\begin{proof}
This is a standard combinatorial procedure. 
We just have to split up the paths from $(0,k)$ to $(n,0)$ 
after the first step that has to be contained in $\mathcal{S}_k$.
Furthermore note that the polynomials 
$xQ_{\ell,j}(x)$, $0\le \ell \le L$, $0\le j \le \min(k,J)$, encode the
multi-stepset $\mathcal{S}_k$ with corresponding weights.
Thus, the system (\ref{eqinfinitesystemsimple}) is satisfied,
which is equivalent to  (\ref{eqoneequationsimple}).
\end{proof}

\subsection{Prime walks}

We again start with polynomials $Q_{\ell,j}(x)$, $0\le \ell \le L$, $0\le j\le J$,
and define multi-sets of steps $\mathcal{S}_k$ 
as in (\ref{eqSk1}) and  (\ref{eqSk2}). The weights $w(s)$ of a step $s$
is defined in the same way, too.
However, instead of considering just allowed lattice paths
from $(k,0)$ to $(n,0)$ we consider now (allowed) lattice paths $p = s_1s_2\cdots s_M$
that connect $(k_1,0)$ to $(n,k_2)$ for $k_1,k_2,n \ge 0$ (with the property that
$s_m \in \mathcal{S}_{k_m}$ if the step $s_m$ starts at level $k_m$, $1\le m \le M$);
the weight $w(p) = w(s_1)w(s_2) \cdots w(s_M)$ is defined accordingly. 
The multi-set of (allowed) paths from $(k_1,0)$ to $(n,k_2)$ is now denoted
by $\mathcal{P}_{n,k_1,k_2}$. 
Furthermore, we denote by $\mathcal{P}_{n,k_1,k_2}^{[\ge k]}$
the sub-multi-set of paths $p \in \mathcal{P}_{n,k_1,k_2}$ 
that contains just paths that stay $\ge k$ and set
\[
F_{k_1,k_2}^{[\ge k]}(x) = \sum_{n\ge 0} \sum_{p \in  \mathcal{P}_{n,k_1,k_2}^{[\ge k]}}  
w(p) \, x^n.
\]
Clearly, we have
\[
F_{k,0}(x) = F_{k,0}^{[\ge 0]}(x).
\]
Further we note that the step multi-sets satisfy $\mathcal{S}_k = \mathcal{S}_J$ for $k\ge J$.
This implies that
\begin{equation}
    F_{k_1,k_2}^{[\ge k]}(x)  = F_{k_1-k+J,k_2-k+J}^{[\ge J]}(x) 
    \label{eqFk1k2k-relation}
\end{equation}
for $k \ge J$.

We now do the lattice path counting problem, that is, finding proper relations
between the functions $F_{k_1,k_2}^{[\ge k]}(x)$, in a different way.
For this purpose, we introduce the so-called {\it prime walks} (see \cite{BanderierLacknerWallner2020}).

We fix some $k\ge 0$ and $i\ge 0$ and consider the multi-set of 
all allowed paths $\mathcal{A}_{n,k,i}$ that
start at $(0,k)$ and end at $(n,k+i)$ for some $n> 0$ and have
the property that they stay strictly above the level~$k+i$ in between, in particular only the last step ends at level~$k+i$.
Furthermore, let 
\[
A_{k,i}(x) = \sum_{n> 0}  \sum_{p \in \mathcal{A}_{n,k,i}}  w(p) \, x^n
\]
denote the corresponding generating functions. For example, $A_{0,0}(x)$ 
corresponds to the excursions from $0$ to $0$.

In the same way, we define for $k\ge 0$, $0\le i \le k$, $n> 0$ the {\it reverse prime walks} 
$\overline{\mathcal{A}}_{n,k,i}$
that start at $(0,k)$ and end at $(n,k-1)$ and have 
the property that they stay strictly above the level~$k$ in between.
The corresponding generating function is given by
\[
\overline A_{k,i}(x) = \sum_{n> 0}  \sum_{p \in \overline{\mathcal{A}}_{n,k,i}}  w(p) \, x^n.
\]
Note that $A_{k,0}(x) = \overline A_{k,0}(x)$. Furthermore, since the step sets satisfy
$\mathcal{S}_k = \mathcal{S}_J$ for $k\ge J$ it follows that
\begin{equation}\label{eqAkstable}
 A_{k,i}(x) = A_{J,i}(x) \quad\mbox{and}\quad \overline A_{k,i}(x) = \overline A_{J,i}(x)
\quad \mbox{for $k\ge J$.}
\end{equation}

We now set
\begin{align*}
    \overline d_k &= \max\{ \ell-j : 0\le \ell \le L,\ 0\le j\le k,\ Q_{\ell,j}(x) \ne 0 \},
    \qquad (k \le J),
\end{align*}
and
\begin{align*}
    \underline d_k &= \max\{ j-\ell : 0\le \ell \le L,\ 0\le j\le k,\ Q_{\ell,j}(x) \ne 0 \}
    \qquad (k \le J).
\end{align*}
Then we certainly have
\begin{align}
    A_{k,i}(x) &=0 \qquad \mbox{for $i> \overline d_k$ and $0\le k < J$},  \nonumber \\
    A_{k,i}(x) &=0 \qquad \mbox{for $i> \overline d_J$ and $k \ge J$}, \label{eqAkzero}\\
    \overline A_{k,i}(x) &=0 \qquad \mbox{for $i> \underline d_k$ and $0\le k < J$}, \nonumber\\
    \overline A_{k,i}(x) &=0 \qquad \mbox{for $i> \underline d_J$ and $k \ge J$}.\nonumber
\end{align}
Furthermore, we have
\begin{align}
    A_{k,\overline d_k}(x) &= x \sum_{\ell-j = \overline d_k} Q_{\ell,j}(x) \qquad (0\le k < J),
     \nonumber \\
    A_{k,\overline d_J}(x) &= x \sum_{\ell-j = \overline d_J} Q_{\ell,j}(x) \qquad ( k \ge J), 
    \label{eqAkknown}\\
    \overline A_{k,\underline d_k}(x) &= x \sum_{\ell-j = -\underline d_k} Q_{\ell,j}(x) 
    \qquad (0\le k < J),  \nonumber \\
    \overline A_{k,\underline d_J}(x) &= x \sum_{\ell-j = -\underline d_J} Q_{\ell,j}(x) 
    \qquad (k\ge J).  \nonumber
\end{align}
Thus, the only unknown prime walk generating functions are
\begin{equation}\label{eqAkunknown}
      A_{k,i}(x), \qquad 0\le k \le J, \ 0\le i < \overline d_k, \qquad
    \overline A_{k,i}(x), \qquad 0\le k \le J, \ 0\le i < \underline d_k
\end{equation}

The next lemma provides a general relation between the 
generating functions $F_{k_1,k_2}^{[\ge k]}(x)$ and the prime walk generating
functions $A_{k,i}(x)$ and $\overline A_{k,i}(x)$.

\begin{lemma}\label{Le2.2}
We have the following relations between the generating functions
$F_{k_1,k_2}^{[\ge k]}(x)$, $A_{k,i}(x)$ and $\overline A_{k,i}(x)$:
\begin{align*}
    F_{k_1,k_1}^{[\ge k]}(x) 
    &= 1 + \sum_{m=k}^{k_1} \overline A_{k_1,k_1-m}(x) F_{m,k_1}^{[\ge k]}(x) 
    \qquad (k\le k_1),\\
    F_{k_1,k_2}^{[\ge k]}(x) 
    &= \sum_{m=k}^{k_1} \overline A_{k_1,k_1-m}(x) F_{m,k_2}^{[\ge k]}(x)
    + \sum_{m=k_1+1}^{k_2}  A_{k_1,m-k_1}(x) F_{m,k_2}^{[\ge m]}(x)
    \qquad (k\le k_1 < k_2), \\
    F_{k_1,k_2}^{[\ge k]}(x) 
    &= \sum_{m=k}^{k_1} \overline A_{k_1,k_1-m}(x) F_{m,k_2}^{[\ge k]}(x)
        \qquad (k\le k_2 < k_1),
\end{align*}
and
\begin{align*}
    A_{k,i}(x) &= \sum_{0\le \ell \le L, 0\le j\le k\, :\, \ell-j = i} x Q_{\ell,j}(x) \\
    & + \sum_{i_1 > i,\, i_2 > i} \ 
    \sum_{0\le \ell_1 \le L, 0\le j_1\le k\, :\, \ell_1-j_1 = i_1} x Q_{\ell_1,j_1}(x) \  
    F_{k + i_1,k + i_2}^{[\ge k+i+1]}(x) \\
    & \qquad \qquad  \qquad \qquad  \sum_{0\le \ell_2 \le L, 0\le j_2\le \min(k + i_2,J)\, :\, \ell_2-j_2 = -(i_2-i)} 
    x Q_{\ell_2,j_2}(x) \\ 
    &\qquad \qquad \qquad (0\le k < J, \ 0\le i \le \overline d_k), \\
     A_{k,i}(x) &= \sum_{0\le \ell \le L, 0\le j\le J\, :\, \ell-j = i} x Q_{\ell,j}(x) \\
    & + \sum_{i_1 > i,\, i_2 > i} \ 
    \sum_{0\le \ell_1 \le L, 0\le j_1\le J\, :\, \ell_1-j_1 = i_1} x Q_{\ell_1,j_1}(x) \  
    F_{k + i_1,k + i_2}^{[\ge k+i+1]}(x)  \\
    \\ & \qquad \qquad \qquad \qquad
    \sum_{0\le \ell_2 \le L, 0\le j_2\le J\, :\, \ell_2-j_2 = -(i_2-i)} x Q_{\ell_2,j_2}(x), \\
    \\ & \qquad \qquad \qquad (k\ge J, \ 0\le i \le \overline d_J),
\end{align*}
as well as 
\begin{align*}
    \overline A_{k,i}(x) &= \sum_{0\le \ell \le L, 0\le j\le k\, :\, \ell-j = -i} x Q_{\ell,j}(x) \\
    & + \sum_{i_1 > 0,\, i_2 > 0} \ 
    \sum_{0\le \ell_1 \le L, 0\le j_1\le k\, :\, \ell_1-j_1 = i_1} x Q_{\ell_1,j_1}(x) \ 
    F_{k + i_1,k + i_2}^{[\ge k+1]}(x) \\
    & \qquad \qquad \qquad \qquad 
    \sum_{0\le \ell_2 \le L, 0\le j_2\le \min( k + i_2,J) \, :\, \ell_2-j_2 = -(i + i_2)} 
    x Q_{\ell_2,j_2}(x), 
    \\ & \qquad \qquad \qquad (0\le k < J, \ 0\le i \le \underline d_k), \\
     \overline A_{k,i}(x) &= \sum_{0\le \ell \le L, 0\le j\le J\, :\, \ell-j = -i} x Q_{\ell,j}(x) \\
    & + \sum_{i_1 > 0,\, i_2 > 0} \ 
    \sum_{0\le \ell_1 \le L, 0\le j_1\le J\, :\, \ell_1-j_1 = i_1} x Q_{\ell_1,j_1}(x) \ 
    F_{k + i_1,k + i_2}^{[\ge k+1]}(x)  \\
    & \qquad \qquad \qquad \qquad 
    \sum_{0\le \ell_2 \le L, 0\le j_2\le J\, :\, \ell_2-j_2 = -(i + i_2)} x Q_{\ell_2,j_2}(x), \\
    \\ & \qquad \qquad \qquad (k\ge J, \ 0\le i \le \underline d_J).
\end{align*}
\end{lemma}

\begin{proof}
The proof of the relations for $F_{k_1,k_2}^{[\ge k]}(x)$ is (again) standard 
(see \cite{BanderierLacknerWallner2020}).
In order to get the relation for $F_{k_1,k_1}^{[\ge k]}(x)$, we split at the first point 
different from the starting point $(0,k_1)$ that reaches a level $\le k_1$.
The same splitting applies for the (third) case $k_1 > k_2$. If $k_1 > k_2$ then there is 
either a point (different from the starting point) that reaches a level $\le k_1$ -- this 
provides the fist term -- or the path stays at levels $> k_1$. In the latter case
let $m \in \{k_1+1, \ldots, k_2\}$ be the minimum level that occurs after the first step.
Here we split at the first occurrence of the level $m$. This gives the second part.

The equation for $A_{k,i}(x)$ can be seen in the following way. We either reach
the endpoint $(n,k+i)$ in one step -- this provides the first term -- or the first
step reaches a level $k+i_1$ for some $i_1 > 0$, then we stay at a level $\ge k+i+1$ until the last step that goes from level $k+i_2 > k+i$ to the final level $k+i$. This 
provides the second term. The argument for $\overline A_{k,i}(x)$ is very similar.
\end{proof}

We note that there are other ways to describe prime walks. 
For example, we can also use the prime-walk decomposition to get a recursive description.

Secondly, it is also possible to split the lattice paths in a way that the prime walk 
is put to the very end. For example, in this way we get the relation
\begin{equation}\label{eqFkrecuresionvariant}
    F_{k_1,k_2}^{[\ge k]}(x) 
    = \sum_{m=k}^{k_2}  F_{k_1,m}^{[\ge k]}(x)  A_{m,k_2-m}(x) 
        \qquad (k\le k_1 < k_2).
\end{equation}

\subsection{A finite grammar}

The main difference between the infinite linear system~\eqref{eqinfinitesystemsimple} and the non-linear system  (from Lemma~\ref{Le2.2}) is that we can find a finite subsystem that contains all the information.

\begin{lemma}\label{Lefinitesystem1}
    The finite polynomial system of equations for the $F_{k_1,k_2}^{[\ge k]}(x)$ (with $0\le k_1 < L+J$, $0\le k_2 < 2J$, $0\le k \le \min(J,k_1,k_2)$)
    $A_{k,i}(x)$ (with $0\le k \le J$, $0\le i < \overline d_k$) and $\overline A_{k,i}(x)$
    (with $0\le k \le J$, $0\le i < \underline d_k$) 
    from Lemma~\ref{Le2.2} is a positive polynomial system 
    of equations that uniquely determines all these functions. Furthermore, all other
    functions $F_{k_1,k_2}^{[\ge k]}(x)$ can be expressed in a polynomial way with
    nonnegative coefficients from them. Thus, all functions $F_{k_1,k_2}^{[\ge k]}(x)$
    correspond to a finite grammar.
\end{lemma}

\begin{proof}
By the above discussion, we know that the only unknown prime walk generating functions
are  $A_{k,i}(x)$  for $0\le k \le J$, $0\le i < \overline d_k$, and 
$\overline A_{k,i}(x)$, $0\le k \le J$, $0\le i < \underline d_k$    (see \ref{eqAkunknown}).
We also note that by  (\ref{eqAkstable}) all appearing prime walk generating functions
can be reduced to these functions; all the other ones are either zero or explicitly
known (see (\ref{eqAkzero}) and (\ref{eqAkknown})).

We study first the equation for $A_{k,i}(x)$
(from Lemma~\ref{Le2.2}), where $0\le k\le J$ and $0\le i \le \overline d_k \le L$.
On the right hand side the function $F_{k+i_1,k+i_2}^{[\ge k+i+1]}(x)$ appears, where
$i_1 \le L$ and $i_2\le i + J$. If $k+i+1 \le J$ then we certainly have
$k+i_1 < J + L$ and $k+i_2 = k+i + i-i_2 < J + J$. However, if $k+i+1 > J$ we have
\[
F_{k+i_1,k+i_2}^{[\ge k+i+1]}(x) = F_{i_1-i+J-1,i_2-i + J-1}^{[\ge J]}(x),
\]
where we (again) have $i_1-i+J-1 < L+J$ and $i_2-i + J-1< 2J$.

Similarly we can study the equation for $\overline A_{k,i}(x)$, where
$0\le k\le J$ and $0\le i \le \underline d_k \le J$.
In this case on the right hand side the function $F_{k+i_1,k+i_2}^{[\ge k+1]}(x)$ appears,
where $i_1 \le L$ and $i_2 +i \le J$.
If $k+1 \le J$ we certainly have $k+i_1 < J + L$ and $k+i_2 < 2J$. 
Secondly, if $k+1 > J$ we have
\[
F_{k+i_1,k+i_2}^{[\ge k+1]}(x) = F_{i_1+J-1,i_2+J-1}^{[\ge J]}(x),
\]
where $i_1+J-1 < L+J$ and $i_2+J-1 < 2J$. 

Summing up, we rewriting the right hand side of the equations for 
$A_{k,i}(x)$ and $\overline A_{k,i}(x)$ the only functions that appear
are of the form $F_{k_1,k_2}^{[\ge k]}$ with $0\le k_1 < L+J$, $0\le k_2 < 2J$, $0\le k \le \min(J,k_1,k_2)$.

Finally we check the equations for $F_{k_1,k_2}^{[\ge k]}(x)$ (from Lemma~\ref{Le2.2}) in this 
range. As already mentioned all appearing prime walk generating functions can be reduced
to the cases $A_{k,i}(x)$  for $0\le k \le J$, $0\le i < \overline d_k$, and 
$\overline A_{k,i}(x)$, $0\le k \le J$, $0\le i < \underline d_k$. Thus, we only have
to take care of the $F$-functions that appear on the right hand side.
In the first case (where $k_1 = k_2$) and in the third case (where $k_2 < k_1$) 
the right hand side contains only functions of the form 
$F_{m,k_2}^{[\ge k]}(x)$ with $k\le m \le k_1$. Thus, we certainly have $m < L+J$ and $k_2 < 2J$
and we are done. In the second case ($k_1 < k_2$) we have to be more careful.
There is no problem for the first sum, where $F_{m,k_2}^{[\ge k]}$ appears (with $k\le m \le k_1)$.
The second sum contains the function $F_{m,k_2}^{[\ge m]}(x)$. If $m\le J$ then we certainly have
$m < L+J$. Conversely if $m> J$ then we have
\[
F_{m,k_2}^{[\ge m]}(x) = F_{J,k_2-m+J}^{[\ge J]}(x),
\]
where $J< L+J$ and $k_2-m + J< k_2 < 2J$. 

This finishes the first part of the proof.

For the second part we first observe that the functions 
\[
\frac 1{1-A_{k,0}(x)} = \frac 1{1-\overline A_{k,0}(x)} \qquad (k\ge 0)
\]
are known, actually they can directly written in terms of the solution 
functions of the above finite system.
In any case (in particular for $k\ge J$) we have
\[
\frac 1{1-A_{k,0}(x)} =  F_{k,k}^{[\ge k]}(x),
\]
and for $k > J$ we also have
\[
\frac 1{1-A_{k,0}(x)} = \frac 1{1-A_{J,0}(x)} = F_{J,J}^{[\ge J]}(x) 
\]

Secondly we use now a slightly modified system for the 
functions $F_{k_1,k_2}^{[\ge k]}(x)$ that follows from 
Lemma~\ref{Le2.2} and (\ref{eqFkrecuresionvariant}):
\begin{align*}
      F_{k_1,k_1}^{[\ge k]}(x) 
    &= 1 + \sum_{m=k}^{k_1} \overline A_{k_1,k_1-m}(x) F_{m,k_1}^{[\ge k]}(x) 
    \qquad (k\le k_1),\\
    F_{k_1,k_2}^{[\ge k]}(x) 
    &= \sum_{m=k}^{k_2}  F_{k_1,m}^{[\ge k]}(x)  A_{m,k_2-m}(x) 
        \qquad (k\le k_1 < k_2), \\
    F_{k_1,k_2}^{[\ge k]}(x) 
    &= \sum_{m=k}^{k_1} \overline A_{k_1,k_1-m}(x) F_{m,k_2}^{[\ge k]}(x)
        \qquad (k\le k_2 < k_1).
\end{align*}
We rewrite this system to:
\begin{align*}
      F_{k_1,k_1}^{[\ge k]}(x) 
    &= \frac 1{1-\overline A_{k_1,0}(x)} \left( 1 + \sum_{m=k}^{k_1-1} \overline A_{k_1,k_1-m}(x) F_{m,k_1}^{[\ge k]}(x) \right)
    \qquad (k\le k_1),\\
    F_{k_1,k_2}^{[\ge k]}(x) 
    &= \frac 1{1-A_{k_2,0}(x)} \left( \sum_{m=k}^{k_2-1}  F_{k_1,m}^{[\ge k]}(x)  A_{m,k_2-m}(x) \right)
        \qquad (k\le k_1 < k_2), \\
    F_{k_1,k_2}^{[\ge k]}(x) 
    &= \frac 1{1-\overline A_{k_1,0}(x)} 
    \left( \sum_{m=k}^{k_1-1} \overline A_{k_1,k_1-m}(x) F_{m,k_2}^{[\ge k]}(x)
    \right)
        \qquad (k\le k_2 < k_1).
\end{align*}
Now, for every given $k\ge 0$, we proceed on induction on the sum $k_1+k_2$. 
The starting point is the case $k_1=k_2 = k$. Here we already observed that
the functions $F_{k,k}^{[\ge k]}(x)$ can be considered as known functions
(that are part of the solution of the above finite system).
The induction step is now trivial since the sum of the indices on the right hand
side is always smaller than the sum $k_1+k_2$ on the left hand side. Furthermore
all prime walk generating functions are either zero, explicit, or part of the
above finite system. This completes the proof of the lemma.
\end{proof}

\subsection{Proof of Theorem~\ref{Th1} for one equation}\label{sec:reduction}

Finally, we consider the general case~\eqref{eqoneequation}, where we represent
the polynomial $P(x,u)$ as
\[
P(x,u) = \sum_{m=0}^M P_m(x) u^m.
\]
By linearity the solution $F(x,u)$ can be represented by
\begin{equation}\label{eqFGrelation}
F(x,u) = \sum_{m=0}^m P_m(x) G_m(x,u),
\end{equation}
%and in particular by
%\[
%F_0(x) = F(x,0) = \sum_{m=0}^m P_m(x) F_m(x,0),
%\]
where $G_m(x,u)$ is the solution of the equation
\begin{equation}\label{eqoneequation-m}
    G_m(x,u) = u^m + x \sum_{\ell = 0}^L Q_{\ell}(x,u) \Delta^\ell G_m(x,u).
\end{equation}
If we write 
\[
G_m(x,u) = \sum_{k\ge 0} G_{m;k}(x) u^k
\]
then (\ref{eqoneequation-m}) rewrites to the infinite system
\begin{align}
    G_{m;0}(x) &=  \delta_{0,m} + x \sum_{\ell = 0}^{L} Q_{\ell,j}(x) G_{m;\ell}(x), \nonumber \\
    G_{m;k}(x) &=  \delta_{k,m} + x \sum_{\ell = 1}^L \sum_{j = 0}^k Q_{\ell,j}(x) G_{m;\ell+k-j}(x), 
    \qquad (0< k < J) \label{eqinfinitesystemsimple-2}\\
        G_{m;k}(x) &=  \delta_{k,m} + x \sum_{\ell = 1}^L \sum_{j = 0}^J Q_{\ell,j}(x)
        G_{m;\ell+J-j}(x) \qquad (k > J). \nonumber 
\end{align}

\begin{lemma}\label{Le2.3}
    We have the relation
    \[
    G_{m;k}(x) = F_{k,m}^{[\ge 0]}(x),
    \]
    that is, the functions $G_{m;k}(x)$ are the generation function of the allowed lattice paths 
    that start at level $k$ and end at level $m$.
\end{lemma}

\begin{proof}
The proof is precisely the same at that of Lemma~\ref{Le2.1}. We
just partition the paths after the first step. Note that the $1 = \delta_{m,m}$ 
appears now from the zero step path from $(0,m)$ to $(0,m)$. 
\end{proof}

Now we can finish the proof of Theorem~\ref{Th1} for one equation.
By Lemma~\ref{Le2.2} and Lemma~\ref{Le2.3} it follows that
the functions $G_{m;k}(x)$ corresponds to a finite grammar.
By (\ref{eqFGrelation}) we have
\[
F_k(x) = \sum_{m=0}^m P_m(x) G_{m;k}(x)
\]
and, thus, $F_k(x)$ corresponds to a finite grammar, too.

\subsection{An example}

Let us consider the example
\begin{equation}\label{eqgeneralexample}
    F(x,u) = 1 + xu F(x,u) +  x(1+u^3) \Delta^1 F(x,u).
\end{equation}
This equation translates into the following infinite system:
\begin{align*}
    F_{0,0}(x) &= 1 + x F_{1,0}(x), \\
    F_{1,0}(x) &= x F_{0,0}(x) + x F_{2,0}(x), \\
    F_{2,0}(x) &= x F_{1,0}(x) + x F_{3,0}(x), \\
    F_{3,0}(x) &= x F_{2,0}(x) + x F_{4,0}(x) + x F_{1,0}(x), \\
    F_{4,0}(x) &= x F_{3,0}(x) + x F_5(x) + x F_{2,0}(x), \\
    &\vdots
\end{align*}
This means that the equations for $F_{k,0}(x)$ are stable for $k\ge 3$ but
have a different scheme for $k< 3$. By using the general notation we have
$L=1$ and $J=3$ and the step (multi-)sets are given by
\begin{align*}
    \mathcal{S}_0 &= \{(1,1)\}, \\
    \mathcal{S}_1 &= \{(1,1), (1,-1)\}, \\
    \mathcal{S}_2 &= \{(1,1), (1,-1)\}, \\
    \mathcal{S}_k &= \{(1,1), (1,-1), (1,-2)\} \qquad (k\ge 3).
\end{align*}
The first observation is that the functions $F_{k_1,k_2}^{[\ge k]}(x)$ are stable
quite early. Actually we have
\[
F_{k_1,k_2}^{[\ge k]}(x) = F_{k_1-k+1,k_2-k+1}^{[\ge 1]} \qquad (k\ge 1).
\]

We start with an analysis of the prime walks. By Lemma~\ref{Le2.2} we have
\[
A_{0,0}(x) = x F_{1,1}^{[\ge 1]}(x) x, \qquad A_{0,1}(x) = x, \qquad A_{1,i}(x) = 0 \quad (i\ge 2)
\]
and
\begin{align*}
    A_{k,0}(x) &= x F_{k+1,k+1}^{[\ge k+1]}(x) x + x F_{k+1,k+2}^{[\ge k+1]}(x) x
     = x F_{1,1}^{[\ge 1]}(x) x + x F_{1,2}^{[\ge 1]}(x) x= A_{1,0}(x)  \quad (k\ge 1), \\
     A_{k,1}(x) &= x \quad (k\ge 1), \qquad A_{k,i}(x) = 0 \quad (k\ge i\ge 2).
\end{align*}
as well as
\begin{align*}
    \overline A_{k,0}(x) &= A_{k,0}(x) \quad (k\ge 0), \\
    \overline A_{1,1}(x) &= x, \qquad  
    \overline A_{k,1}(x) = x + x  F_{k+1,k+1}^{[\ge k+1]}(x) x 
    = x + x  F_{1,1}^{[\ge 1]}(x) x  = \overline A_{2,1} \quad (k\ge 2),\\
    \overline A_{k,2}(x) &= \overline A_{3,2}(x) = x \quad (k\ge 3), \qquad
    \overline A_{k,i}(x) = 0 \quad (k\ge i\ge 3).
\end{align*}
This means that the only unknown prime walk functions are
\[
A_{0,0}(x),\ A_{1,0}(x),\ \overline A_{2,1}(x),
\]
where (only) the functions
\[
F_{1,1}^{[\ge 1]}(x), F_{1,2}^{[\ge 1]}(x)
\]
appear on the right hand of the equation. For these two functions we have
(again by Lemma~\ref{Le2.2} and by the above relations for the prime walk functions)
\begin{align*}
    F_{1,1}^{[\ge 1]}(x) &= 1 + \overline A_{1,0}(x) F_{1,1}^{[\ge 1]}(x)
    = 1 +  A_{1,0}(x) F_{1,1}^{[\ge 1]}(x), \\
    F_{1,2}^{[\ge 1]}(x) &= \overline A_{1,0}(x) F_{1,2}^{[\ge 1]}(x) 
    +  A_{1,1}(x) F_{2,2}^{[\ge 2]}(x) = 
     A_{1,0}(x) F_{1,2}^{[\ge 1]}(x) 
    +  x F_{1,1}^{[\ge 1]}(x).
\end{align*}
Already this system of $3+2$ equations is solvable. Actually we can start with
the strongly connected system of $3$ equations for 
$A_{1,0}(x)$, $F_{1,1}^{[\ge 1]}(x)$, $F_{1,2}^{[\ge 1]}(x)$:
\begin{align*}
    A_{1,0}(x) &= x F_{1,1}^{[\ge 1]}(x) x + x F_{1,2}^{[\ge 1]}(x) x, \\
        F_{1,1}^{[\ge 1]}(x) &=1 +  A_{1,0}(x) F_{1,1}^{[\ge 1]}(x), \\
    F_{1,2}^{[\ge 1]}(x) &= 
     A_{1,0}(x) F_{1,2}^{[\ge 1]}(x) 
    +  x F_{1,1}^{[\ge 1]}(x).
\end{align*}
By general theory (see \cite{Drmota2009}) it is known that all three solution functions
have a common dominating squareroot singularity $x_0 > 0$.
This holds then also for $A_{0,0}(x)$ and $\overline A_{2,1}(x)$.

However, it is not clear that holds for all the other functions.
For example, from 
\[
F_{0,0}^{[\ge 0]}(x) = 1 + \overline A_{0,0}(x) F_{0,0}^{[\ge 0]}(x)
\]
it follows that
\[
F_{0,0}^{[\ge 0]}(x) = \frac 1{1- \overline A_{0,0}(x) }.
\]
If $\overline A_{0,0}(x_0) \ge 1$ then $F_{0,0}^{[\ge 0]}(x)$ has a different
kind of dominating singularity. 

But there is a simple trick that shows that such a situation cannot occur.
By using the observation that the step sets are weakly increasing from one level to the next,
it follows that  
\[
[x^n] F_{1,1}^{[\ge 1]}(x) \ge [x^n] F_{0,0}^{[\ge 0]}(x) \quad\mbox{and}\quad
[x^n] A_{1,i}(x) \ge [x^n] A_{0,i}(x)
\]
for all $n\ge 0$. In particular we can write
\begin{align*}
    F_{1,1}^{[\ge 1]}(x) &=  F_{0,0}^{[\ge 0]}(x)  + D_1(x), \\
    A_{1,0}(x) &=  A_{0,0}(x) + B_1(x).
\end{align*}
where $D_1(x), B_1(x)$ are power series with nonnegative coefficients.

Due to the monotonicity of the step sets, there is also a monotonicity property in the corresponding equations. That is, the equations for $F_{1,1}^{[\ge 1]}(x), A_{1,0}(x), A_{1,1}(x)$ contain at least the corresponding terms as the equations for $F_{0,0}^{[\ge 0]}(x), A_{0,0}(x), A_{0,1}(x)$, respectively. 
For example, the equation $A_{1,0}(x) = x F_{1,1}^{[\ge 1]}(x) x + x F_{1,2}^{[\ge 1]}(x)$ has one more term than the equation $A_{0,0}(x) = x F_{1,1}^{[\ge 1]}(x) x$.
This implies that we split the above system of equations in the following way:
\begin{align*}
    F_{0,0}^{[\ge 0]}(x)  + D_1(x) &= 1 +  (A_{0,0}(x) + B_1(x))(F_{0,0}^{[\ge 0]}(x)  + D_1(x)), \\
    A_{0,0}(x) + B_1(x) &= x (F_{0,0}^{[\ge 0]}(x)  + D_1(x)  ) x + xF_{1,2}^{[\ge 2]}(x)x .
\end{align*}
Due to the monotonicity property and by comparing with the two equations
\begin{align*}
    F_{0,0}^{[\ge 0]}(x) &= 1 +  A_{0,0}(x) F_{0,0}^{[\ge 0]}(x), \\ 
    A_{0,0}(x) &= x F_{1,1}^{[\ge 1]}(x) x = x(F_{0,0}^{[\ge 0]}(x)  + D_1(x))x,
\end{align*}
we obtain new equations for $D_1(x),B_1(x)$:
\begin{align*}
    D_1(x) &= B_1(x) F_{0,0}^{[\ge 0]}(x) + (A_{0,0}(x) + B_1(x)) D_1(x), \\
    B_1(x) &=  xF_{1,2}^{[\ge 2]}(x)x.
\end{align*}
The only unknown function here is $F_{1,2}^{[\ge 2]}(x)$. By adding the corresponding
equation
\[
F_{1,2}^{[\ge 1]}(x) =      A_{1,0}(x) F_{1,2}^{[\ge 1]}(x)     +  x F_{1,1}^{[\ge 1]}(x) =
   (A_{0,0}(x) + B_1(x) ) F_{1,2}^{[\ge 1]}(x)     +  x (F_{0,0}^{[\ge 0]}(x) + D_1(x)),
\]
we end up with a strongly connected positive system of equations that contains
the function $F_{0,0}^{[\ge 0]}(x)$ as well:
\begin{align*}
    F_{0,0}^{[\ge 0]}(x) &= 1 +  A_{0,0}(x) F_{0,0}^{[\ge 0]}(x), \\ 
    A_{0,0}(x) &=  x(F_{0,0}^{[\ge 0]}(x)  + D_1(x))x, \\
    D_1(x) &= B_1(x) F_{0,0}^{[\ge 0]}(x) + (A_{0,0}(x) + B_1(x)) D_1(x), \\
    B_1(x) &=  xF_{1,2}^{[\ge 2]}(x)x, \\
    F_{1,2}^{[\ge 1]}(x) &=   
   (A_{0,0}(x) + B_1(x) ) F_{1,2}^{[\ge 1]}(x)     +  x (F_{0,0}^{[\ge 0]}(x) + D_1(x)).
\end{align*}
Thus, it follows that all these functions have a common square-root singularity at 
the dominant singularity at $x_0 > 0$. 

There are now several ways to show that all appearing generating functions
have the same dominating square-root singularity. In this example there is a simple
and direct way. 
Since $F_{1,1}^{[\ge 1]}(x) =  F_{0,0}^{[\ge 0]}(x)  + D_1(x)$ and $A_{1,0}(x) =  A_{0,0}(x) + B_1(x)$
the same holds for $F_{1,1}^{[\ge 1]}(x)$ and $A_{1,0}(x)$. 
Furthermore, since $\overline A_{2,1}(x) = x + x(F_{0,0}^{[\ge 0]}(x) + D_1(x))x$ we get also
the same property for $\overline A_{2,1}(x)$. Now we can proceed as in the last part of the
proof of Lemma~\ref{Lefinitesystem1} and show inductively that all functions
$F_{k_1,k_2}^{[\ge k]}(x)$ have a dominating square-root singularity at $x_0$.

\subsection{Proof of Theorem~\ref{Th2} for one equation}\label{sec:proofTh2single}

Now we use the assumption that from every level $k_1$ we can reach every level $k_2$ 
(either the same or any other) by an allowed lattice path that stays $\ge k$, where $k\ge \min(k_1,k_2)$. 
This is equivalent to the statement that the (infinite) system of equations 
(\ref{eqinfinitesystemsimple}) is strongly connected. (If $F_{k,0}(x)$ depends (after some iterations)
on $F_{0,0}(x)$ then there has to be an allowed path from level $k$ to level $0$. Similarly if
$F_{0,0}(x)$ depends (after some iterations) on $F_{k,0}(x)$ then there has to be an 
allowed path from level $0$ to level $k$. By combining such paths we can start at any level $k_1$
and reach with an allowed path every level $k_2$. Due to monotonicity property and by shifting 
a path accordingly we can also assume that we stay $\ge k$, where $k\ge \min(k_1,k_2)$.)
%This condition has to be properly rephrased for Theorem~\ref{Th2}.

We recall that we already know that
all appearing generating functions correspond to a finite grammar. 
Thus, the kind of singularities (as well as the asymptotic behavior of the coefficients)
is quite restricted, see \cite{BanderierDrmota2015}.

The starting point of the proof of Theorem~\ref{Th2} is the following observation on 
the singular behavior of the functions $F_{k,k}^{[\ge k]}(x)$.

\begin{lemma}\label{Le2.6-1}
    Suppose that the assumptions of Theorem~\ref{Th2} are satisfied. Then 
    the functions $F_{k,k}^{[\ge k]}(x)$, $k\ge 0$, have a common dominating
    square-root singularity at some $x_0>0$.
\end{lemma}

\begin{proof}
We recall that it is sufficient to consider the cases $0\le k \le J$. 
For $k > J$ we have $F_{k,k}^{[\ge k]}(x) = F_{J,J}^{[\ge J]}(x)$.
Furthermore by the monotonicity property of the step multi-sets $\mathcal{S}_k$ we 
can write
\begin{equation}\label{eqdecomp1}
      F_{k,k}^{[\ge k]}(x) = F_{0,0}^{[\ge 0]}(x) + D_{1}(x) + \cdots + D_{k}(x), \qquad (k\le J),
\end{equation}
where the functions $D_j(x)$, $1\le j\le J$, have nonnegative coefficients.
Similarly we can represent 
\begin{equation}\label{eqdecomp2}
  A_{k,0}(x) = A_{0,0}(x) + B_{1}(x) + \cdots + B_{k}(x),, \qquad (k\le J),
\end{equation}
where the functions $B_j(x)$, $1\le j\le J$, have nonnegative coefficients.
Since we have the functional equations
\[
    F_{k,k}^{[\ge k]}(x) = 1 + A_{k,0}(x) F_{k,k}^{[\ge k]}(x) \quad (k\le J)
\]
it follows (by taking differences and by using the above representations) that
\begin{align}\label{eqDjrep}
    D_j(x) &= B_j(x) \left( F_{0,0}^{[\ge 0]}(x) + D_1(x) + \cdots + D_{j-1}(x)   \right) \\
    & + \left( A_{0,0}(x) + B_{1}(x) + \cdots + B_{j}(x)  \right) D_j(x), \qquad (1\le j \le J).
    \nonumber
\end{align}
In particular, (\ref{eqDjrep}) shows that $D_j(x) \ne 0$ if and only if $B_j(x) \ne 0$
(note that by assumption $A_{0,0}(x)\ne 0$ and $F_{0,0}^{[\ge 0]}(x)\ne 0$).
Thus, we can restrict ourselves to those $j$ for which $D_j(x) \ne 0$. 
Let $J'$ the maximum $j\in \{1,2,\ldots, J\}$ for which $D_j(x) \ne 0$.
Then the equation for $D_{J'}(x)$ contains on the right hand side the 
functions $A_{0,0}(x)$, $ F_{0,0}^{[\ge 0]}(x)$ and all functions
$B_j(x)$, $D_j(x)$ with $1\le j\le J$ that are non-zero. 
Furthermore, the equation $F_{0,0}^{[\ge 0]}(x) = 1 + A_{0,0}(x) F_{0,0}^{[\ge 0]}(x)$
as well as the equations for $D_j(x)$ $1\le j \le J$, for which $D_j(x) \ne 0$, contains
on the right hand side $A_{0,0}(x)$, $F_{0,0}^{[\ge 0]}(x)$. 

Similarly to the decompositions (\ref{eqdecomp1}) and (\ref{eqdecomp2}) we can
decompose all other functions $F_{k_1,k_2}^{[\ge k]}(x)$, $A_{k,i}(x)$, $\overline A_{k,i}(x)$
and due to the monotonicity property of the step multi-sets $\mathcal{S}_k$ 
we can replace the old system for $F_{k_1,k_2}^{[\ge k]}(x)$, $A_{k,i}(x)$, $\overline A_{k,i}(x)$
by a new one that is again a positive polynomial system. 
By Lemma~\ref{Lefinitesystem1} it is sufficient to consider a finite (old) system of equations.
Hence, it is also sufficient to consider a finite new system (where -- without loss of 
generality -- we remove zero equations).

The first goal is to find a strongly connected system of equations that computes
$F_{0,0}^{[\ge 0]}(x)$. For this purpose we consider the corresponding dependency
graph $G$ and restrict ourselves to the strongly connected component $G$ 
that contains $F_{0,0}^{[\ge 0]}(x)$. 
By the above observations all functions $D_j(x)$ $1\le j \le J$, for which $D_j(x) \ne 0$, as
well as $A_{0,0}(x)$ and all functions $B_j(x)$ $1\le j \le J$, for which $D_j(x) \ne 0$,
are contained in this subsystem. Hence, all theses functions have a common
dominant square-root singularity $x_0> 0$. Consequently the functions 
\[
      F_{k,k}^{[\ge k]}(x) = F_{0,0}^{[\ge 0]}(x) + D_{1}(x) + \cdots + D_{k}(x), \qquad (k\le J),
\]
and as well the functions $F_{k,k}^{[\ge k]}(x) = F_{J,J}^{[\ge J]}(x)$, $k\ge J$, have this
dominant square-root singularity $x_0> 0$. 
\end{proof}

The next lemma is a general property for generating functions that correspond to 
a finite grammar. By \cite{BanderierDrmota2015} the dominant singularities can be only of
special (dyadic) types $(1-x/x_0)^\alpha$. Furthermore, 
the asymptotics of the coefficients -- at least if
we restrict ourselves to a proper residue class $n \equiv r \bmod m$ -- correspond to the 
dominant singularity, that is $c x_0^{-n} n^{-\alpha-1}$ 
(with some $c> 0$). We say that a generating function $f(x)$ (of this type) 
with leading asymptotics $c x_0^{-n} n^{-\alpha-1}$, $n \equiv r \bmod m$,  is dominated by $g(x)$ 
with leading asymptotics $\overline c \overline x_0^{-n} n^{-\overline \alpha-1}$, 
$n \equiv \overline r \bmod \overline m$, if
$x_0^{-n} n^{-\alpha-1} = O\left( \overline x_0^{-n} n^{-\overline \alpha-1}  \right)$, that is,
$\overline x_0 < x_0$ or $\alpha \le \overline \alpha$ (if $\overline x_0 = x_0$).

\begin{lemma}\label{Le2.6-2}
    Suppose that $f(x)$ and $g(x)$ are non-polynomial generating functions that correspond
    to a finite grammar such that
    \[
    g(x) = P(x) f(x) + R(x),
    \]
    where $P(x)$ is a non-zero polynomial with nonnegative coefficients and 
    $R(x)$ a power series with nonnegative coefficients. 
    Then $f(x)$ is dominated by $g(x)$.
\end{lemma}

\begin{proof}
First we observe that $P(x)f(x)$ has the same type of singularity as $f(x)$.
Furthermore, since
\[
[x^n] g(x) \ge [x^n] P(x) f(x) 
\]
it directly follows that $f(x)$ is dominated by $g(x)$. 
\end{proof}

The next observation will then conclude the proof of Theorem~\ref{Th2}.

\begin{lemma}\label{Le6.2-3}
    Suppose that the assumptions of Theorem~\ref{Th2} are satisfied.
    Then all functions $F_{k,k}^{[\ge k]}(x)$ have the same dominating singularity
    as $F_{0,0}^{[\ge 0]}(x)$. Furthermore the prime walk functions 
    $A_{k,i}(x)$, $\overline A_{k,i}(x)$ are either polynomials or have again
    the same dominating singularity as $F_{0,0}^{[\ge 0]}(x)$.
\end{lemma}

\begin{proof}
By Lemma~\ref{Le2.6-1} all functions $F_{k,k}^{[\ge k]}(x)$, $0\le k \le J$, have
the same dominating (square-root) singularity (namely that of $F_{0,0}^{[\ge 0]}(x)$).

We consider now a function $F_{k_1,k_2}^{[\ge k]}(x)$ with $k\le \min(k_1,k_2)$. By assumption there is an
allowed path $p_1$ from level $k_1$ to level $k$ (that stays at levels $\ge k$ and goes $n_1$ units 
into the $x$-direction) as well an
allowed path $p_2$ from level $k$ to level $k_2$ (that stays at levels $\ge k$, too, and
goes $n_2$ units into the $x$-direction). Then we certainly have
\[
F_{k_1,k_2}^{[\ge k]}(x) = w(p_1)x^{n_1} F_{k,k}^{[\ge k]}(x) w(p_2) x^{n_2} + r(x)
\]
for some power series $r(x)$ with nonnegative coefficients. We just observe that all allowed paths
from level $k_1$ to level $k_2$ that stay at levels $\ge k$ contain the concatenation of
$p_1$, the paths corresponding to $F_{k,k}^{[\ge k]}(x)$, and $p_2$.
Thus, $F_{k,k}^{[\ge k]}(x)$ is dominated by $F_{k_1,k_2}^{[\ge k]}(x)$.

Similarly, we can concatenate $p_2$, the paths corresponding to $F_{k_1,k_2}^{[\ge k]}(x)$, and $p_1$
and obtain a part of all paths corresponding to $F_{k,k}^{[\ge k]}(x)$. In other terms we have
\[
F_{k,k}^{[\ge k]}(x) = w(p_2)x^{n_2} F_{k_1,k_2}^{[\ge k]}(x) w(p_1) x^{n_1} + \overline r(x)
\]
with a power series $\overline r(x)$ with nonnegative coefficients. 
Hence, $F_{k_1,k_2}^{[\ge k]}(x)$ is dominated by $F_{k,k}^{[\ge k]}(x)$. 

Putting these two observations together it follows that $F_{k_1,k_2}^{[\ge k]}(x)$ and $F_{k,k}^{[\ge k]}(x)$
have the same kind of dominating (square-root) singularity.

Finally, by applying the relations from Lemma~\ref{Le2.2} it follows that
the prime walk functions $A_{k,i}(x)$, $\overline A_{k,i}(x)$ are either polynomials in $x$
(with nonnegative coefficients) or polynomials (again with nonnegative coefficients) in 
$x$ and finitely many functions $F_{k_1,k_2}^{[\ge k]}(x)$. Thus, the
functions $A_{k,i}(x)$, $\overline A_{k,i}(x)$ are either polynomials or have
the same dominating (square-root) singularity as  $F_{0,0}^{[\ge 0]}(x)$.
\end{proof}

\section{Several catalytic equations} \label{sec:several}

We suppose now that we have a system of linear catalytic equations of the form~\eqref{eqcatalyticsystem}. 
Similarly to the above procedure, we translate these equations into a corresponding lattice-path problem. 

Again, we start with a simplified version, where we assume that 
$P_{s}(x,u) = 1$:
\begin{align}\label{eqcatalyticsystemsimplified}
    F_s(x,u) &= 1 + x \sum_{t=1}^d \sum_{\ell =0}^L Q_{s,t,\ell} (x,u) \Delta^\ell F_t(x,u), \qquad (1\le s \le d) \\
    & = 1 + x \sum_{t=1}^d \sum_{\ell =0}^L \sum_{j=0}^J Q_{s,t,\ell,j} (x) u^j \Delta^\ell F_t(x,u)
\end{align} 
Clearly this system translates into an infinite system ($1\le s\le d$, $k\ge 0$):
\begin{equation}\label{eqcatalyticsystemsimplified-2}
    F_{s;k,0}(x) = \delta_{k,0} +
x \sum_{t=1}^d  \sum_{\ell = 0}^L \sum_{m=0}^{\min(J,k)} Q_{s,t,\ell, j}(x) F_{t;k+\ell-j,0}(x),
\qquad (1\le s \le d,\, k\ge 0).
\end{equation}

In contract to the case of one equation we have to distinguish between
$d$ types (or ``colors'') for lattice paths. In particular we have $d^2$ sets
of step multi-sets $\mathcal{S}_{s,t,k}$, $1\le i,j\le d$, $0\le k \le J$, that are defined in the following way:
\begin{align}
\mathcal{S}_{s,t,k}  &= \bigcup_{\ell = 0}^L \bigcup_{j = 0}^{\min(k,J)}
\{ (1+r,\ell-j) : 0\le r \le {\rm deg}Q_{s,t,\ell,j}, [x^r] Q_{\ell,j}(x) \ne 0   \} 
\qquad (0\le k \le J)  \label{eqSstk}
\end{align}
Observe that the multi-sets $\mathcal{S}_{s,t,k}$, $1\le s,t,\le d$, might contain
a specific step several times, that is, they have to be labeled accordingly.
In order to simplify the notation we do not make this labels explicit.

Furthermore we define weights $w$ of a step $(1+r,\ell-j) \in \mathcal{S}_{s,t,k}$ by
\begin{align*}
w(1+r,\ell-j) &= [x^r] Q_{s,t,\ell,j}(x).
\end{align*}
We now say that a path $p = s_1\cdots s_M$ (of at least one step) that starts at level $k$ 
is of type $s$ and is allowed if the first step $s_1$ is contained in $\mathcal{S}_{s,t_0}$ for some $t_0$, and if consecutive steps are consistent in the following way:
after the first step $s_1$ that is contained in $S_{s,t_0,k}$ for some $t_0$ and
connects level $k$ to level $k_1$ we shall have $s_{2} \in S_{t_0,t_1,k_1}$ for some $t_1$ 
and some $k_1$ that connects level $k_1$ to level $k_2$, 
$s_{3} \in S_{t_1,t_2,k_2}$ for some $t_2$ and some $k_2$ etc.
And clearly the weight $w(p)$ of an allowed path $p = s_1\cdots s_M$ is defined by
\[
w(p) = \prod_{m=1}^M w(s_m).
\]
For the empty walk with zero steps, that is of every type, we set (as usual) $w(\emptyset) = 1$. 

As in the case of a single catalytic equation, we now define as $\mathcal{P}_{s;n,k}$ the multi-set
of allowed  paths of type $s$ from $(0,0)$ to $(n,k)$ that stay above the $x$-axis.
(By convention, the empty path is contained in $\mathcal{P}_{s;0,0}$ for $1\le s \le d$.)
Furthermore, if we set 
\[
F_{s;k,0}(x) = \sum_{n\ge 0}  \sum_{p\in \mathcal{P}_{s;n,k}} w(p) x^n \qquad (1\le s \le d,\, k\ge 0)
\]
then it is immediate that these functions satisfy the infinite system~\eqref{eqcatalyticsystemsimplified-2}. 
We just have to split the paths after the first step.

Finally, we show (very similarly to the case of one catalytic equation) that there is also a prime walk decomposition. For this purpose, we define
prime walks $\mathcal{A}_{s,t;k,i,n}$ of type $(s,t)$, $1\le s,t\le d$, as
all paths from $(0,k)$ to $(n,k+i)$, where the first step is in $S_{s,t_0,k}$ for some $t_0$, where
the last step is in $S_{s_0,t,k_1}$ for some $s_1$ and $k_1> k+i$ and connects 
level $k_1$ with level $k$, and where the path stays above the level~$k$ 
in between. The corresponding generating functions are defined by
\[
A_{s,t;k,i}(x) = \sum_{n> 0} \sum_{p\in \mathcal{A}_{s,t;k,i,n}} w(p) x^n.
\]
Similarly, we define reverse prime walks  $\overline{\mathcal{A}}_{s,t;k,i,n}$ that connect $(0,k)$ 
to $(n,k-i)$ and their generating functions $\overline A_{s,t;k,i}(x)$. 

We set
\begin{align*}
    \overline d_{s,t,k} &= \max\{ \ell-j : 0\le \ell \le L,\ 0\le j\le k,\ Q_{s,t,\ell,j}(x) \ne 0 \},
    \qquad (k \le J),
\end{align*}
and
\begin{align*}
    \underline d_{s,t,k} &= \max\{ j-\ell : 0\le \ell \le L,\ 0\le j\le k,\ Q_{s,t\ell,j}(x) \ne 0 \}
    \qquad (k \le J).
\end{align*}
Then we certainly have
\begin{align}
    A_{s,t;k,i}(x) &=0 \qquad \mbox{for $i> \overline d_{s,t,k}$ and $0\le k < J$},  \nonumber \\
    A_{s,t;k,i}(x) &=0 \qquad \mbox{for $i> \overline d_{s,t,J}$ and $k \ge J$}, \label{eqAkzero-2}\\
    \overline A_{s,t;k,i}(x) &=0 \qquad \mbox{for $i> \underline d_{s,t,k}$ and $0\le k < J$}, \nonumber\\
    \overline A_{s,t;k,i}(x) &=0 \qquad \mbox{for $i> \underline d_{s,t,J}$ and $k \ge J$}.\nonumber
\end{align}
Furthermore we have
\begin{align}
    A_{s,t,k,\overline d_k}(x) &= x \sum_{\ell-j = \overline d_k} Q_{s,t,\ell,j}(x) \qquad (0\le k < J),
     \nonumber \\
    A_{s,t;k,\overline d_J}(x) &= x \sum_{\ell-j = \overline d_J} Q_{s,t,\ell,j}(x) \qquad ( k \ge J), 
    \label{eqAkknown}\\
    \overline A_{s,t;k,\underline d_k}(x) &= x \sum_{\ell-j = -\underline d_k} Q_{s,t,\ell,j}(x) 
    \qquad (0\le k < J),  \nonumber \\
    \overline A_{s,t;k,\underline d_J}(x) &= x \sum_{\ell-j = -\underline d_J} Q_{s,t,\ell,j}(x) 
    \qquad (k\ge J).  \nonumber
\end{align}

We consider now a slightly more refined counting problem of lattice paths.
We define as $\mathcal{P}_{s,t;n,k_1,k_2}^{[\ge k]}$ the set
of allowed paths from $(0,k_1)$ to $(n,k_2)$ that stay at level $\ge k$, 
where the first step is in $S_{s,t_0,k_1}$ for some $t_0$ and where
the last step is in $S_{s_0,t,k'}$ for some $s_0$ and some $k'$ that connects
level $k'$ and level $k_2$. 
The corresponding generating functions are 
\[
F_{s,t;k_1,k_2}^{[\ge k]}(x) = \sum_{n\ge 0}  \sum_{p\in \mathcal{P}_{s,t;n,k_1,k_2}^{[\ge k]}} 
w(p)\, x^n \qquad (1\le s,t, \le d,\, k_1,k_2 \ge 0,\, k\ge \min(k_1,k_2)),
\]
and as in Lemma~\ref{Le2.2} we have 
\begin{align*}
    F_{s,t;k_1,k_1}^{[\ge k]}(x)
    &= 1 + \sum_{t_0=1}^d \sum_{m=k}^{k_1} \overline A_{s,t_0;k_1,k_1-m}(x) F_{t_0,t;m,k_1}^{[\ge k]}(x) 
    \qquad (k\le k_1),\\
    F_{s,t;k_1,k_2}^{[\ge k]}(x) 
    &= \sum_{t_0=1}^d \sum_{m=k}^{k_1} \overline A_{s,t_0;k_1,k_1-m}(x) F_{t_0,t;m,k_2}^{[\ge k]}(x)
    + \sum_{t_0=1}^d \sum_{m=k_1+1}^{k_2}  A_{s,t_0;k_1,m-k_1}(x) F_{t_0,t;m,k_2}^{[\ge m]}(x)
    \\
    &= \sum_{t_0=1}^d \sum_{m=k_1}^{k_2} F_{s,t_0;k_1,m}^{[\ge k]}(x) A_{t_0,t;m,k_2-m}(x) 
      \qquad (k\le k_1 < k_2),   \\
    F_{s,t;k_1,k_2}^{[\ge k]}(x) 
    &= \sum_{t_0=1}^d \sum_{m=k}^{k_1} \overline A_{s,t_0;k_1,k_1-m}(x) F_{t_0,t;m,k_2}^{[\ge k]}(x)
        \qquad (k\le k_2 < k_1),
\end{align*}
and
\begin{align*}
    A_{s,t;k,i}(x) &= \sum_{0\le \ell \le L, 0\le j\le k\, :\, \ell-j = i} x Q_{s,t,\ell,j}(x) \\
    & + \sum_{t_0=1}^d \sum_{t_1=1}^d \sum_{i_1 > i,\, i_2 > i} \ 
    \sum_{0\le \ell_1 \le L, 0\le j_1\le k\, :\, \ell_1-j_1 = i_1} x Q_{s,t_0\ell_1,j_1}(x) \  
    F_{t_0,t_1;k + i_1,k + i_2}^{[\ge k+i+1]}(x) \\
    & \qquad \qquad  \qquad \qquad  
    \sum_{0\le \ell_2 \le L, 0\le j_2\le \min(k + i_2,J)\, :\, \ell_2-j_2 = -(i_2-i)} 
    x Q_{t_1,t,\ell_2,j_2}(x) \\ 
    &\qquad \qquad \qquad (0\le k < J, \ 0\le i \le \overline d_{s,t,k}), \\
     A_{s,t,k,i}(x) &= \sum_{0\le \ell \le L, 0\le j\le J\, :\, \ell-j = i} x Q_{s,t,\ell,j}(x) \\
    & + \sum_{t_0=1}^d \sum_{t_1=1}^d \sum_{i_1 > i,\, i_2 > i} \ 
    \sum_{0\le \ell_1 \le L, 0\le j_1\le J\, :\, \ell_1-j_1 = i_1} x Q_{s,t_0\ell_1,j_1}(x) \  
    F_{t_0,t_1;k + i_1,k + i_2}^{[\ge k+i+1]}(x)  \\
    \\ & \qquad \qquad \qquad \qquad
    \sum_{0\le \ell_2 \le L, 0\le j_2\le J\, :\, \ell_2-j_2 = -(i_2-i)} x Q_{t_1,t\ell_2,j_2}(x), \\
    \\ & \qquad \qquad \qquad (k\ge J, \ 0\le i \le \overline d_{s,t,J}),
\end{align*}
as well as 
\begin{align*}
    \overline A_{s,t;k,i}(x) &= \sum_{0\le \ell \le L, 0\le j\le k\, :\, \ell-j = -i} x Q_{s,t,\ell,j}(x) \\
    & + \sum_{t_0=1}^d \sum_{t_1=1}^d \sum_{i_1 > 0,\, i_2 > 0} \ 
    \sum_{0\le \ell_1 \le L, 0\le j_1\le k\, :\, \ell_1-j_1 = i_1} x Q_{s,t_0\ell_1,j_1}(x) \ 
    F_{t_0,t_1;k + i_1,k + i_2}^{[\ge k+1]}(x) \\
    & \qquad \qquad \qquad \qquad 
    \sum_{0\le \ell_2 \le L, 0\le j_2\le \min( k + i_2,J) \, :\, \ell_2-j_2 = -(i + i_2)} 
    x Q_{t_1,t,\ell_2,j_2}(x), 
    \\ & \qquad \qquad \qquad (0\le k < J, \ 0\le i \le \underline d_{s,t,k}), \\
     \overline A_{s,t;k,i}(x) &= \sum_{0\le \ell \le L, 0\le j\le J\, :\, \ell-j = -i} x Q_{s,t,\ell,j}(x) \\
    & + \sum_{t_0=1}^d \sum_{t_1=1}^d \sum_{i_1 > 0,\, i_2 > 0} \ 
    \sum_{0\le \ell_1 \le L, 0\le j_1\le J\, :\, \ell_1-j_1 = i_1} x Q_{s,t_0,\ell_1,j_1}(x) \ 
    F_{t_0,t_1;k + i_1,k + i_2}^{[\ge k+1]}(x)  \\
    & \qquad \qquad \qquad \qquad 
    \sum_{0\le \ell_2 \le L, 0\le j_2\le J\, :\, \ell_2-j_2 = -(i + i_2)} x Q_{t_0,t,\ell_2,j_2}(x), \\
    \\ & \qquad \qquad \qquad (k\ge J, \ 0\le i \le \underline d_{s,t,J}).
\end{align*}

This provides again an infinite system of equations for the 
generating function $F_{s,t;k_1,k_2}^{[\ge k]}(x)$, $A_{s,t,k,i}(x)$, and $\overline A_{s,t,k,i}(x)$.
As in the case of a single equation we can consider a finite subsystem for the function
\begin{align*}
     F_{s,t;k_1,k_2}^{[\ge k]}(x)&: \quad   1\le s,t, \le d, \ 0\le k_1 < L+J, \  0\le k_2 < 2J, \ 0\le k \le \min(J,k_1,k_2), \\
     A_{s,t;k,i}(x)&: \quad   1\le s,t, \le d, \  0\le k \le J, \ 0\le i < \overline d_{s,t,k}, \\
     \overline A_{s,t;k,i}(x)&: \quad    1\le s,t, \le d,\ 0\le k \le J, \ 0\le i < \underline d_{s,t,k}, 
\end{align*}
that can be independently solved. Again, all other functions can be expressed 
in a positive polynomial way with these functions (compare Lemma~\ref{Lefinitesystem1}).
This completes the proof of Theorem~\ref{Th1}.

For the proof of Theorem~\ref{Th2} we apply a similar procedure as described in
Section~\ref{sec:proofTh2single}. The modifications are (again) obvious.

Note that we use the assumption that for every pair $(s,t) \in \{1,\ldots, d\}^2$ 
we can reach from every level $k_1$ every level $k_2$ by an allowed lattice path 
of type $(s,t)$ that stays $\ge k$, where $k\ge \min(k_1,k_2)$.
This is equivalent to the statement that the (infinite) system of equations 
(\ref{eqcatalyticsystemsimplified-2}) is strongly connected.
%This condition has to be properly rephrased for Theorem~\ref{Th2}.

\medskip\noindent
{\bf Acknowledgement.} 
The authors are very thankful to Gilles Schaeffer for several discussions on the topic of this paper.

%---------------------------- Bibliography -------------------------------

% You can also use BibTex, of course
%{\small
%\bibliographystyle{plain}
%\bibliography{CyrilMichael25}
%}

\end{document}